\documentclass[11pt,a4paper]{amsart}
\usepackage{amsthm,amsmath,amssymb}
\usepackage[english]{babel}

\newtheorem{thm}{Theorem}[section]
\newtheorem{prop}[thm]{Proposition}
\newtheorem{lem}[thm]{Lemma}
\newtheorem{cor}[thm]{Corollary}

\theoremstyle{definition}
\newtheorem{definition}[thm]{Definition}

\theoremstyle{remark}

\numberwithin{equation}{section}

\begin{document}
\title{On Asymptotic Property C}
\author{Trevor Davila}

\begin{abstract}
The notion of asymptotic property $C$ was introduced in \cite{dranishnikov}. We show that asymptotic property $C$ is preserved by finite products. We also show that countable restricted direct products of countable groups with finite asymptotic dimension have asymptotic property $C$. Then we introduce hyperbolic property $C$, an infinite dimensional version of hyperbolic dimension.
\end{abstract}

\maketitle

\section{Introduction}
Asymptotic dimension was introduced by Mikhail Gromov in \cite{gromov} to study finitely-generated groups, but this notion applies to all metric spaces. In \cite{yu1} Guoliang Yu proved the Novikov higher signature conjecture for groups with finite asymptotic dimension. In \cite{yu2} Yu introduced property A, a dimension-like property weaker than finite asymptotic dimension, and proved the coarse Baum-Connes conjecture for groups with property A. This stimulated interest in large scale dimension-like properties which are satisfied by spaces with infinite asymptotic dimension and which imply property A. Such infinite dimensional properties include asymptotic property C \cite{dranishnikov}, finite decomposition complexity \cite{guentnertesserayu}, and straight finite decomposition complexity \cite{dranishnikovzarichnyi}. A basic question is which constructions preserve these properties. In section 3 we show that asymptotic property C is preserved by finite direct products, answering a question from \cite{banakh} and \cite{bellmoran}. We also show that asymptotic property C holds for countable restricted direct products of groups with finite asymptotic dimension. This can be viewed as a partial affirmative answer to question 3.2 from \cite{yamauchi}.

In \cite{dranishnikovzarichnyi} the following implications were demonstrated for metric spaces:
$$\textit{asymptotic property C} \implies sFDC \implies \textit{property A}$$
In section 4 we introduce the notions of hyperbolic property C and weak hyperbolic property C, which are dimension-like properties for spaces with infinite hyperbolic dimension. These satisfy the following implications for metric spaces:
\begin{equation}
\begin{aligned}
\textit{asymptotic property C} & \implies \textit{hyperbolic property C} \implies \textit{weak hyperbolic property C} \\
& \implies sFDC \implies \textit{property A} \notag
\end{aligned}
\end{equation}

The author has been made aware that the preservation of asymptotic property C by finite direct products was shown independently in \cite{bellnagorko}. The author would like to thank Alexander Dranishnikov for many stimulating conversations during the writing of this paper.

\section{Preliminaries}
Let $X$ be a metric space. For nonempty $A,B \subset X$, we let $d(A,B) = \text{inf}\{ d(a,b) \, | \, a \in A, \, b \in B \}$.

Let $R > 0$. A family $\mathcal{U}$ of nonempty subsets of $X$ is $R$-disjoint if $d(A,B) > R$ for all $A,B \in \mathcal{U}$ with $A \neq B$.

A family $\mathcal{U}$ of subsets of $X$ is uniformly bounded if $\text{mesh} \, \mathcal{U} = \text{sup} \{ \text{diam}(U) \, : \, U \in \mathcal{U} \} < \infty$.

The asymptotic dimension of $X$ does not exceed $n$, written $\text{asdim} X \leq n$, if for every $R > 0$ there are uniformly bounded $R$-disjoint families $\mathcal{U}_0, \mathcal{U}_1, \dots , \mathcal{U}_n$ of subsets of $X$ such that $\bigcup_{i=0}^n \mathcal{U}_i$ is a cover of $X$. We say $X$ has finite asymptotic dimension if there exists $n$ with $\text{asdim} \, X \leq n$.

For a family of metric spaces $\{ X_\alpha \}$, we say $\text{asdim} \, \{ X_\alpha \} \leq n$ if for every $R > 0$ there exists $D \geq 0$ such that for every $\alpha$ there exist $R$-disjoint families $\mathcal{U}_0 , \dots \mathcal{U}_n$ such that $\text{mesh} \, \mathcal{U}_i \leq D$ for $0 \leq i \leq n$ and $\bigcup_{i=0}^n \mathcal{U}_i$ is a cover of $X_\alpha$.

We recall the definition of asymptotic property $C$ from \cite{dranishnikov}.

\begin{definition}
A metric space $X$ has \textit{asymptotic property C} if for every sequence $R_0 \leq R_1 \leq R_2 \dots $ of positive reals there exists $n \geq 0$ and uniformly bounded $R_i$-disjoint families $\mathcal{U}_i$, $i = 0, \dots , n$ such that $\bigcup_{i=0}^n \mathcal{U}_i$ is a cover of $X$.
\end{definition}

For metric spaces $X,Y$, a function $f:X \to Y$ is a coarse embedding if there exist nondecreasing functions $\rho_-, \rho_+ : [0, + \infty )$ with
$$\lim_{t \to + \infty} \rho_-(t) = \lim_{t \to + \infty} \rho_+(t) = + \infty$$
and
$$\rho_-(d(x,y)) \leq d(f(x),f(y)) \leq \rho_+(d(x,y))$$
for all $x,y \in X$.

The following is easy to check.
\begin{prop}
For metric spaces $X,Y$ with a coarse embedding $f: X \to Y$, if $Y$ has asymptotic property $C$, then $X$ has asymptotic property $C$.
\end{prop}

Given a sequence of pointed metric spaces $(X_i, x_i)_{i=1}^\infty$, we define its restricted product $\times_{i=1}^\infty X_i$ to be the set of sequences $(a_i)_{i=1}^\infty$ with each $a_i \in X_i$ and $a_i = x_i$ for all but finitely many $i$. We equip $\times_{i=1}^\infty X_i$ with the metric
$$d((a_i), (b_i)) = \sum_{i=1}^\infty i \cdot d(a_i, b_i) $$

In \cite{smith} it was shown that any two proper, left-invariant metrics on a countable group are coarse equivalent. Given a sequence of countable groups $(G_i)_{i=1}^\infty$ each equipped with a proper, left-invariant metric, we choose the identity as the base point of each $G_i$. Then the restricted product $\times_{i=1}^\infty G_i$ is a countable group with the operation which applies the group operation of each $G_i$ coordinate-wise, and clearly the metric defined above is proper and left-invariant.

A tree is a connected acyclic graph. We view a tree $T$ as a metric space by equipping its set of vertices with the shortest path metric, i.e. for $u,v \in T$ we say $d(u,v)$ is the length of the shortest path in $T$ from $u$ to $v$.

For a tree $T$ with root $e$ and $0 \leq m < n$, we define the annulus from $m$ to $n$:

$$[m,n) = \{ v \in T : m \leq d(v,e) < n \}$$

Let $\mathcal{X}, \mathcal{Y}$ be metric families and $R>0$. We say $\mathcal{X}$ is \textit{$R$-decomposable over $\mathcal{Y}$}, denoted $\mathcal{X} \stackrel{R}{\longrightarrow} \mathcal{Y}$, if for any $X \in \mathcal{X}$, $X = \bigcup (\mathcal{U}_1 \cup \mathcal{U}_2)$, where $\mathcal{U}_1$ and $\mathcal{U}_2$ are $R$-disjoint families such that $\mathcal{U}_1 \cup \mathcal{U}_2 \subset \mathcal{Y}$.
We recall the definition of straight finite decomposition complexity from \cite{dranishnikovzarichnyi}.

\begin{definition}
A metric family $\mathcal{X}$ has \textit{straight finite decomposition complexity} (sFDC) if for every sequence $R_0 \leq R_1 \leq \dots$ of positive reals there exists $n \in \mathbb{N}$ and metric families $\mathcal{U}_i$, $i = 0, \dots , n$ such that $\mathcal{U}_n$ is uniformly bounded and
$$\mathcal{X} \stackrel{R_{0}}{\longrightarrow} \mathcal{U}_{0} \stackrel{R_{1}}{\longrightarrow} \mathcal{U}_1 \stackrel{R_{2}}{\longrightarrow} \dots \stackrel{R_{n}}{\longrightarrow} \mathcal{U}_n$$
\end{definition}

For a metric space $X$, $r > 0$ and $x \in X$, let $B_r(x)$ denote the open ball in $X$ of radius $r$ centered at $x$.

A subset $U \subset X$ of a metric space $X$ is $(N, R)$-\textit{large scale doubling} if for every $x \in X$ and every $r \geq R$, $B_{2r}(x) \cap U$ can be covered by $N$ balls of radius $r$ with centers in $X$.

A family $\mathcal{U}$ of subsets of a metric space $X$ is \textit{uniformly large scale doubling} if there exists $(N,R)$ such that each $U \in \mathcal{U}$ is $(N,R)$-large scale doubling, and such that every finite union of sets in $\mathcal{U}$ is $(N,R')$-large scale doubling for some $R' > 0$.

A family $\mathcal{U}$ of subsets of a metric space $X$ is \textit{weakly uniformly large scale doubling} if there exists $(N,R)$ such that each $U \in \mathcal{U}$ is $(N,R)$-large scale doubling.

\section{Asymptotic Property C}
\begin{thm}
If $X$ and $Y$ are metric spaces with asymptotic property $C$, then $X \times Y$ has asymptotic property $C$.
\end{thm}
\begin{proof}
Fix a sequence $R_1 , R_2 , \dots$ of positive real numbers. Use a bijection \newline
$f: \mathbb{N} \to \mathbb{N} \times \mathbb{N}$ to reindex this sequence as $\{ R_{m,n} : m,n \in \mathbb{N} \}$. Since $Y$ has asymptotic property $C$, for each $m$ there is $n(m) \in \mathbb{N}$ such that there exist $\mathcal{V}^m_1, \mathcal{V}^m_2, \dots , \mathcal{V}^m_{n(m)}$ collections of subsets of $Y$ such that each $\mathcal{V}^m_n$ is $R_{m,n}$-disjoint and uniformly bounded, and 
$$\bigcup_{n=1}^{n(m)} \mathcal{V}^{m}_n$$
is a cover of $Y$. For each $m$, let $R'_m = \text{max}(R_{m,1}, R_{m, 2}, \dots, R_{m, n(m)} )$

Now we apply the asymptotic property $C$ of $X$ to the sequence $R'_1, R'_2, \dots$ to obtain a sequence $\mathcal{U}_1, \mathcal{U}_2, \dots \mathcal{U}_k$ of collections of subsets of $X$ such that each $\mathcal{U}_m$ is $R'_m$-disjoint and uniformly bounded, and
$$\bigcup_{m=1}^{k} \mathcal{U}_m$$
is a cover of $X$.

For each $m \leq k$ and $n \leq n(m)$, let
$$\mathcal{W}_{m,n} = \{ U \times V : U \in \mathcal{U}_m \text{ and } V \in \mathcal{V}^m_n \}$$
Clearly each $\mathcal{W}_{m,n}$ is uniformly bounded and $R_{m,n}$-disjoint, and 
$$\bigcup_{m \leq k, \, n \leq n(m)} \mathcal{W}_{m,n}$$ is a cover of $X \times Y$. By reindexing using $f^{-1}$, we get indices $n_1 < n_2 < \dots < n_l$ and families $\mathcal{W}_{n_1}, \mathcal{W}_{n_2}, \dots , \mathcal{W}_{n_l}$. For each $n < n_l$ with $n \neq n_i$ for all $i < l$, let $\mathcal{W}_n = \emptyset$. Then each $\mathcal{W}_n$ is $R_n$-disjoint and uniformly bounded, and
$$\bigcup_{j=1}^{n_l} \mathcal{W}_j$$
is a cover of $X \times Y$. We conclude that $X \times Y$ has asymptotic property $C$.
\end{proof}

\begin{lem}
For any $R > 0$ and any $R$-disjoint family $\mathcal{V}$ of annuli with uniformly bounded width in a tree $T$, $\mathcal{V}$ has a uniformly bounded $R$-disjoint refinement.
\end{lem}
\begin{proof}
Let $w = \text{sup}\{ \text{width}(V) : V \in \mathcal{V} \}$. The family of annuli in $\mathcal{V}$ distance at most $R$ from $e$ is uniformly bounded, so it is enough to consider the annuli which are distance $>R$ from $e$. So let $[a, b) = I \in \mathcal{V}$ with $a > R$. For $x,y \in I$, say $x \, E \, y$ if $(x | y)_e \geq a - R$. Note that $E$ is an equivalence relation since $T$ is a tree. If $x E y$, then
$$d(x,y) = d(x,e) + d(y,e) - 2(x|y)_e < 2b - 2(a - R) = 2b - 2a + 2R \leq 2w + 2R$$
On the other hand, if $x, y$ are not $E$-related, then
$$d(x,y) = d(x,e) + d(y,e) - 2(x|y)_e > 2a - 2(a - R) = 2R$$
Hence the $E$-classes of $I$ are uniformly bounded and $R$-disjoint.
\end{proof}
Given a tree $T$ and an $R$-disjoint family of annuli $\mathcal{V}$ in $T$, we denote the refinement obtained from the above lemma by $\text{Ref}_R(\mathcal{V})$. The following theorem makes use of an argument analogous to that of the main theorem of \cite{yamauchi}.

\begin{thm}
If $(T_i)_{i=1}^\infty$ is a sequence of trees, then $\times_{i=1}^{\infty} T_i$ has asymptotic property C.
\end{thm}
\begin{proof}
Let $R_0 < R_1 < R_2 < \dots $ be arbitrary positive reals. Choose $k, m \in \mathbb{N}$ such that $R_0 < k$ and $R_{k2^k} < m$. Each $T_i$ has asymptotic dimension at most $1$, so we have bounded and $R_{k2^k}$-disjoint families $\mathcal{V}_0^i , \mathcal{V}_1^i$ for $1 \leq i \leq k + m$ such that $\mathcal{V}_0^i \cup \mathcal{V}_1^i$ covers $T_i$. For each $i$, replace each set in $\mathcal{V}_1^i$ with its intersection with the complement of $\cup \mathcal{V}_0^i$ so we have $(\cup \mathcal{V}_0^i) \cap (\cup \mathcal{V}_1^i) = \emptyset$. Note that $\times_{i=1}^\infty T_i$ admits a coarse embedding into a restricted product of trees each with at least $2$ vertices, so we may assume that each $V_0^i$ and each $V_1^i$ is nonempty.

For each $i \in \mathbb{N}$, let $[a,b)_i$ denote the annulus in $T_i$ from $a$ to $b$. Let $S = R_0 + R_{k2^k}$. Fix a bijection $\psi : \{1,2, \dots , 2^m\} \to \{0,1\}^m$. For each $i \in \{1, 2, \dots , k \}$ and $l \in \{1, 2, \dots , 2^m \}$, let
$$C^i_l = \text{Ref}_{R_0} ( \{ [0, (2^m + l)S - R_0)_i \} \cup \{ [(2^mn + l)S, (2^m(n+1) + l)S - R_0)_i : n \in \mathbb{N} \} )$$
$$D^i_l = \text{Ref}_{R_{k2^k}}\{ [(2^mn + l)S - R_0, (2^mn + l)S)_i : n \in \mathbb{N} \}$$
$$W_l = \bigg \{ \prod_{i=k+1}^{k+m} V_i : V_i \in \mathcal{V}^i_{\psi(l)_{i - k}} \bigg \}$$
Note that for each $i$ and $l$, $\mathcal{C}^i_l$ is $R_0$-disjoint and $\mathcal{C}^i_l \cup \mathcal{D}^i_l$ covers $T_i$. Also, for each $i$ we have that $\bigcup_{l=1}^{2^m}\mathcal{D}^i_l$ is $R_{k2^k}$-disjoint, each $\mathcal{W}_l$ is $R_{k2^k}$-disjoint, $\bigcup_{l=1}^{2^m}\mathcal{W}_l$ is disjoint, and $\bigcup_{l=1}^{2^m}\mathcal{W}_l$ covers $\prod_{i=k+1}^{k+m}T_i$.

Now fix a bijection $\varphi : \{ 1, 2, \dots , 2^k \} \to \{0,1\}^k$, let
$$\mathcal{U}_0 = \Bigg \{ \prod_{i=1}^k C_i \times W \times \prod_{i > k + m} \{ x_i \} : (C_1, C_2, \dots , C_k , W ) \in \bigcup_{l = 1}^{2^m} \big ( \prod_{i = 1}^k \mathcal{C}^i_l \times \mathcal{W}_l \big ) , (x_i) \in \times_{i=k+m+1}^\infty T_i \Bigg \} $$
For each $s \in \{1, 2, \dots , k \}$ and $t \in \{ 1, 2, \dots , 2^k \}$ let
\begin{equation}
\begin{aligned}
\mathcal{U}_{2^k(s-1)+t} & = \Bigg \{ \prod_{i=1}^{s-1} V_i \times D \times \prod_{i=s+1}^k V_i \times W \times \prod_{i > k + m}  \{ x_i \} : \\
V_i & \in V^i_{\varphi(t)_i}, i \in \{ 1, \dots , s-1, s+1, \dots , k \}, (D,W) \in \bigcup_{l = 1}^{2^m}(\mathcal{D}^s_l \times \mathcal{W}_l), (x_i) \in \times_{i=k+m+1}^\infty T_i \Bigg \} \notag
\end{aligned}
\end{equation}
The collection of all $\mathcal{C}^i_l, \mathcal{D}^i_l$, and $\mathcal{W}^i_l$ is uniformly bounded, so clearly the $\mathcal{U}_j$, $0 \leq j \leq k2^k$ are uniformly bounded.

To see that $\mathcal{U}_0$ is $R_0$ disjoint, let $U = C_1 \times C_2 \times \dots \times C_k \times W \times \prod_{i >k+m} \{ x_i \}$ and $U'= C_1' \times C_2' \times \dots \times C_k' \times W' \times \prod_{i >k+m} \{ x'_i \} $ be distinct sets in $\mathcal{U}_0$, and let $a \in U, b \in U'$. If $(x_i)$ and $(x'_i)$ are distinct, then $a_i \neq b_i$ for some $i > k+m$, so $d((a_i), (b_i)) > k > R_0$. If $W$ and $W'$ are distinct, then they are disjoint, so $a_i \neq b_i$ for some $i > k$. Otherwise $W = W'$, so we have $W, W' \in W_l$ for some $1 \leq l \leq 2^m$. Hence $C_i, C'_i \in \mathcal{C}^i_l$, but $C_i \neq C_i'$ for some $i$, so $d((a_i), (b_i)) > R_0$ since $\mathcal{C}^i_l$ is $R_0$-disjoint for each $i$.

Each $\mathcal{U}_j$ for $1 \leq j \leq k2^k$ can be shown to be $R_{k2^k}$-disjoint along similar lines. It is also not hard to see that the $\mathcal{U}_j$ cover $\times_{i=1}^\infty T_i$.
\end{proof}

\begin{thm}
Suppose $(G_i)_{i=1}^\infty$ is a sequence of countable groups each of finite asymptotic dimension. Then $\times_{i=1}^\infty G_i$ has asymptotic property $C$.
\end{thm}
\begin{proof}
By Theorem A from \cite{dranishnikov2}, every countable group of asymptotic dimension $n$ admits a coarse embedding into a product of $n+1$ trees equipped with the $\text{sup}$ metric. Hence there exists a sequence of trees $(T_k)_{k=1}^\infty$ and a strictly increasing function $K : \mathbb{N} \to \mathbb{N}$ such that for each $i \in \mathbb{N}$ we have a coarse embedding $f_i : G_i \to \prod_{k = K(i)}^{K(i+1) - 1} T_k$, where $\prod_{k = K(i)}^{K(i+1) - 1} T_k$ is equipped with the sup metric. Since each $G_i$ has bounded geometry, we may assume that each $f_i$ is injective. For each $i$ fix proper nondecreasing maps
$$\rho_-^i, \rho_+^i : [0, \infty) \to [0, \infty), \lim_{t \to \infty}\rho_-^i(t) = \infty$$
witnessing that $f_i$ is a coarse embedding. We may assume for every $t > 0$ that
\newline $\rho_+^i(t) \leq \rho_+^j(t)$ for all $i \leq j$. Let $F : \times_{i=1}^\infty X_i \to \times_{k=1}^\infty T_k$ be defined by 
$$F((x_i)_{i=1}^\infty) = (f_i(x_i))_{i=1}^\infty$$
Note that each $f_i(x_i) \in \prod_{k = K(i)}^{K(i+1) - 1} T_k$, so we view $(f_i(x_i))_{i=1}^\infty$ as an element of $\times_{i=1}^\infty T_i$ by appending the terms.
Now let $x, y \in \times_{i=1}^\infty G_i$ and let $M = K(d(x,y) + 1)$. Then $x_i = y_i$ for all $i \geq M$, so
\begin{equation}
\begin{aligned}
d(F(x), F(y)) & \leq \sum_{i=1}^\infty (K(i+1)-1)(K(i+1) - K(i)) d(f_i(x_i), f_i(y_i)) \\
& \leq M^2 \sum_{i=1}^M d(f_i(x_i), f_i(y_i)) \\
& \leq M^2 \sum_{i=1}^M \rho_+^i (d(x_i, y_i)) \\
& \leq M^2 \sum_{i=1}^M \rho_+^M (d(x_i, y_i)) \\
& \leq M^3 \rho_+^M(d(x,y)) \notag
\end{aligned}
\end{equation}
Since the last expression only depends on $d(x, y)$, we can make this the definition of $\rho_+$. Since each $f_i$ is injective, we can assume $\rho^i_-(t) \geq 1$ if $t \geq 1$. Let $\rho_-$ be defined by
$$\rho_-(t) = \text{min} \Big ( \sum_{i=1}^\infty i \cdot \rho_-^i(d(x_i, y_i)) : x,y \in \times_{i=1}^\infty G_i, \, d(x, y) \geq t \Big )$$
Note that $\rho_-$ is nondecreasing, and $\rho_-(d(x, y)) \leq d(F(x), F(y))$ for all $x,y \in \times_{i=1}^\infty G_i$. Suppose $\rho_-$  is not proper, i.e. that there exists an integer $m$ and a sequence of integers $n_1 < n_2 < \dots $ with $\rho_-(n_j) \leq m$ for every $j$. If $x,y \in \times_{i=1}^\infty G_i$ with $x_i \neq y_i$ for some $i > m$, then $d(x_i, y_i) \geq 1$, so $i \cdot \rho_-^i(d(x_i, y_i)) \geq i > m$. Hence for each $n_j$ we can fix $x^j, y^j \in \times_{i=1} G_i$ such that $x^j$ and $y^j$ agree past $m$,
$$d(x^j, y^j) \geq n_j$$
and the minimum in the definition of $\rho_-$ is attained by $x^j,y^j$. Since the pairs $x^j, y^j$ agree past $m$ and $d(x^j, y^j) \to \infty$ as $j \to \infty$, for some $i < m$ the distances $d(x^j_i, y^j_i)$ increase without bound. But $\rho_-^i(d(x^j_i, y^j_i)) \leq \rho(d(x^j, y^j)) \leq m$ for each $j$, contradicting the properness of $\rho_-^i$. Hence $\rho_-$ is proper, and $\rho_-, \rho_+$ witness that $F$ is a coarse embedding. Hence $\times_{i=1}^\infty G_i$ has asymptotic property $C$.
\end{proof}

\section{Hyperbolic Property C}
Buyalo and Schroeder introduced hyperbolic dimension in \cite{buyaloschroeder}, and Cappadocia introduced weak hyperbolic dimension in \cite{cappadocia}. We introduce infinite-dimensional versions of these properties.

\begin{definition}
A metric space $X$ has \textit{hyperbolic property C} if for every sequence $R_0, R_1, \dots$ of positive reals there exists $n \geq 0$ and $R_i$-disjoint families $\mathcal{U}_i$, $i = 0, \dots , n$ such that $\bigcup_{i=0}^n \mathcal{U}_i$ is a uniformly large scale doubling cover of $X$.
\end{definition}

\begin{definition}
A metric space $X$ has \textit{weak hyperbolic property C} if for every sequence $R_0, R_1, \dots$ of positive reals there exists $n \geq 0$ and $R_i$-disjoint families $\mathcal{U}_i$, $i = 0, \dots , n$ such that $\bigcup_{i=0}^n \mathcal{U}_i$ is a weakly uniformly large scale doubling cover of $X$.
\end{definition}

The following is immediate from the definitions.

\begin{prop}
Let $X$ be a metric space. If $X$ has asymptotic property $C$, then $X$ has hyperbolic property $C$. If $X$ has hyperbolic property $C$, then $X$ has weak hyperbolic property $C$.
\end{prop}

\begin{prop}
If $\{ X_\alpha : \alpha \in A \}$ is a family of metric spaces with finite asymptotic dimension uniformly, then $\{ X_\alpha : \alpha \in A \}$ has sFDC.
\end{prop}
\begin{proof}
Suppose $\text{asdim} \, \{ X_\alpha \} = n$, and consider a sequence of positive reals $R_0 < R_1 < \dots$ Then there exists $D > 0$ such that, for each $\alpha$ there exist $\mathcal{U}^\alpha_0, \mathcal{U}^\alpha_1, \dots , \mathcal{U}^\alpha_n$ such that $\mathcal{U}^\alpha_i$ is $R_n$-disjoint,  $\text{mesh} \, \mathcal{U}^\alpha_i \leq D$ for each $i= 0 , \dots , n$, and $\bigcup_{i=0}^n \mathcal{U}_i$ is a cover of $X_\alpha$. We define families $\mathcal{V}^\alpha_i$ for each $\alpha$ and $\mathcal{V}_i$ inductively for $0 \leq i \leq n$. Let
$$\mathcal{V}^\alpha_0 = \{ X_\alpha \cap U : U \in \mathcal{U}^\alpha_0 \}$$
and
$$\mathcal{V}_0 = \Big ( \bigcup_{\alpha \in A} \mathcal{V}_0^\alpha \Big ) \cup \{ X_\alpha \setminus \cup \mathcal{V}^\alpha_0 : \alpha \in A \}$$
For $0 < k < n$, having defined $\mathcal{V}_i^\alpha, \mathcal{V}_i$ for $i < k$, let
$$V = $$
$$\mathcal{V}^\alpha_k = \{ ( X_\alpha \setminus \bigcup \bigcup_{i=0}^{k-1} \mathcal{V}^\alpha_{i} ) \cap U : U \in \mathcal{U}^\alpha_k \}$$
and
$$\mathcal{V}_k = \Big ( \bigcup_{\alpha \in A} \bigcup_{i=0}^k \mathcal{V}^\alpha_i \Big ) \cup \{ X_\alpha \setminus \bigcup \bigcup_{i=0}^k \mathcal{V}^\alpha_i : \alpha \in A\}$$

Finally, let
$$\mathcal{V}_n = \{ ( X_\alpha \setminus \bigcup \bigcup_{i=0}^{n-1} \mathcal{V}^\alpha_{i} ) \cap U : \alpha \in A, \, U \in \mathcal{U}^\alpha_n \} \cup \bigcup_{\alpha \in A} \bigcup_{i=0}^{n-1} \mathcal{V}^\alpha_i$$
Then $\mathcal{V}_n$ is a subfamily of $\bigcup_{\alpha \in A} \bigcup_{i=0}^n \mathcal{U}^\alpha_i$, so $\mathcal{V}_n$ is uniformly bounded. Also
$$\{ X_\alpha \} \stackrel{R_{0}}{\longrightarrow} \mathcal{V}_{0} \stackrel{R_{1}}{\longrightarrow} \mathcal{V}_1 \stackrel{R_{2}}{\longrightarrow} \dots \stackrel{R_{n}}{\longrightarrow} \mathcal{V}_n$$
since each $\mathcal{U}^\alpha_i$ is $R_i$-disjoint. We conclude that $\{ X_\alpha \}$ has sFDC.
\end{proof}

\begin{thm}
If $X$ is a metric space with weak hyperbolic property $C$, then $X$ has $sFDC$.
\end{thm}
\begin{proof}
Let $X$ be a metric space with weak hyperbolic property $C$. Given a sequence $R_0 < R_1 < R_2 < \dots$, there exists $n \in \mathbb{N}$ and $\mathcal{U}_0, \mathcal{U}_1, \mathcal{U}_2, \dots \mathcal{U}_n$ with each $\mathcal{U}_i$ an $R_i$-disjoint family of subsets of $X$, such that $\bigcup_{i=0}^n \mathcal{U}_i$  is a weakly large scale doubling cover of $X$. Let
$$\mathcal{V}_0 = \mathcal{U}_0 \cup \{ X \setminus (\cup \, \mathcal{U}_0) \}$$
Clearly $\{ X \} \stackrel{R_0}{\longrightarrow} \mathcal{V}_0$. Next we $R_1$-decompose $\mathcal{V}_0$ by intersecting $X \setminus (\cup \mathcal{U}_0)$ with the sets in $\mathcal{U}_1$ and so on. In the end we have $\mathcal{V}_n$ a subfamily of $\bigcup_{i=0}^n \mathcal{U}_i$. Hence $\mathcal{V}_n$ is weakly uniformly large scale doubling.
By Proposition 4.5 in \cite{nicasrosenthal}, $\mathcal{V}_n$ has finite asymptotic dimension uniformly. Hence by Proposition 4.4 there is $m \geq 0$ and a sequence of families $\mathcal{V}_{n+1}, \mathcal{V}_{n+2}, \dots, \mathcal{V}_{n+m}$ such that $\mathcal{V}_{n+i} \stackrel{R_{n+i+1}}{\longrightarrow} \mathcal{V}_{n+i+1}$ for every $i < m$, and $\mathcal{V}_{n+m}$ is uniformly bounded. Hence $X$ has sFDC.
\end{proof}

The above theorem and theorem 3.4 from \cite{dranishnikovzarichnyi} together yield the following:

\begin{cor}
If $X$ is a metric space with weak hyperbolic property $C$, then $X$ has Property $A$.
\end{cor}


\medskip

\begin{thebibliography}{9}

\bibitem[BBGRZ]{banakh}
T. Banakh, B. Bokalo, I. Guran, T. Radul, M. Zarichnyi,
\textit{Problems from the Lviv topological seminar},
in E. Pearl (Ed.), Open Problems in Topology II, Elsevier (2007), 655-667.

\bibitem[BM]{bellmoran}
G. Bell, D. Moran,
\textit{On constructions preserving the asymptotic topology of metric spaces},
North Carolina Journal of Mathematics and Statistics, 1 (2015), (46-57).

\bibitem[BN]{bellnagorko}
G. Bell, A. Nag\'{o}rko, \textit{On stability of asymptotic property C for products and some group extensions}.

\bibitem[BS]{buyaloschroeder}
S. Buyalo and V. Schroeder,
\textit{Elements of asymptotic geometry},
EMS Monographs in Mathematics, European Mathematical Society, Z{\"u}rich (2007).

\bibitem[C]{cappadocia}
C. Cappadocia,
\textit{Large scale dimension theory of metric spaces},
Ph.D. Thesis, McMaster University (2014).

\bibitem[D1]{dranishnikov}
A. Dranishnikov,
\textit{Asymptotic topology},
Russian Math. Surveys 55 (2000), no. 6, 1085-1129.

\bibitem[D2]{dranishnikov2}
A. Dranishnikov,
\textit{On hypersphericity of manifolds with finite asymptotic dimension},
Trans. Amer. Math. Soc. 355 (2003), 155-167.

\bibitem[DZ]{dranishnikovzarichnyi}
A. Dranishnikov, M. Zarichnyi,
\textit{Asymptotic dimension, decomposition complexity, and Haver's property C},
Topology Appl. 169 (2014), 99-107.

\bibitem[G]{gromov}
M. Gromov,
\textit{Asymptotic invariants of infinite groups},
in Geometric Group Theory, vol. 2, Cambridge University Press (1993).

\bibitem[GTY]{guentnertesserayu}
E. Guentner, R. Tessera, G. Yu,
\textit{A notion of geometric complexity and its application to topological rigidity},
Invent. Math. 189 (2012), 315-357.

\bibitem[NR]{nicasrosenthal}
A. Nicas, D. Rosenthal,
\textit{Hyperbolic dimension and decomposition complexity},
to appear in London Mathematical Society Lecture Notes Series.

\bibitem[S]{smith}
J. Smith,
\textit{On asymptotic dimension of countable Abelian groups},
Topol. Appl. 153 (2006), 2047-2054.

\bibitem[Y]{yamauchi}
T. Yamauchi,
\textit{Asymptotic property C of the countable direct sum of the integers}, Topol. Appl. 184 (2015), 50-53.

\bibitem[Yu1]{yu1}
G. Yu,
\textit{The Novikov conjecture for groups with finite asymptotic dimension},
Ann. of Math. (2) \textbf{147} (1998), no.2, 325-355.

\bibitem[Yu2]{yu2}
G. Yu,
\textit{The coarse Baum-Connes conjecture for spaces which admit a uniform embedding into Hilbert space},
Invent. Math. \textbf{139} (2000), 201-240.


\end{thebibliography}
\end{document}